\documentclass[11pt]{article}
\usepackage{amsmath, amssymb, amsthm}
\usepackage{bbm}
\usepackage{enumitem}
\usepackage{hyperref}
\usepackage{tikz}
\usetikzlibrary{calc}
\usepackage{xcolor}
\usepackage{multirow}
\usepackage{comment}

\newtheorem{theorem}{Theorem}[section]
\newtheorem{lemma}{Lemma}[section]

\newtheorem{proposition}{Proposition}[section]
\newtheorem{remark}{Remark}

\newcommand{\N}{\mathbb{N}}
\newcommand{\R}{\mathbb{R}}

\usepackage[left=1in, right=1in, bottom = 1in, top = 1in]{geometry}
\title{On the structure of the sandpile identity element on Sierpi\'nski gasket graphs}
\author{Robin Kaiser, Ecaterina Sava-Huss, Julia Überbacher}
\date{\today}

\begin{document}
\maketitle
\begin{abstract}
We consider the identity of the abelian sandpile group of finite approximation graphs of the Sierpi\'nski gasket, and we show that the second-order term in the scaling limit converges to the path distance to the nearest corner on the Sierpi\'nski gasket. The proof relies on a decomposition of the identity of the sandpile group into the sum of a constant function and the Laplacian of the graph distance on the approximating graphs.
\end{abstract}


\section{Introduction}

Bak, Tang and Wiesenfeld \cite{BTW87, BTW88} introduced the sandpile automaton on lattices as a system that displays self-organized criticality: it evolves into a critical state without the fine-tuning of external parameters. Dhar \cite{D90} generalized this model to arbitrary finite graphs and noticed that the dynamics of the model are abelian, which led him to refer to it as the abelian sandpile model. Because of its relevance in physics and its interesting connections to many different fields in mathematics, the abelian sandpile model has since attracted the interest of physicists and mathematicians alike.

For a finite graph $G=(V \cup \{s\}, E)$ with a designated sink vertex $s$, a sandpile configuration is a collection of indistinguishable particles on the vertices of $G$. We call a vertex stable, if the number of particles is smaller than its degree, and unstable otherwise. The unstable vertices can be legally toppled by sending one particle to each neighboring vertex. During a toppling, if particles are sent to the sink, they are lost forever. By using these toppling dynamics, one can define the abelian sandpile Markov chain on the space of stable sandpile configurations in the following manner. At each time step, a particle is added at a uniformly random location, and subsequently, the configuration is stabilized by performing all possible legal topplings.

Research on the abelian sandpile model spans topics ranging from the limit shape in the single-source model — where all particles are initially placed on a single vertex and the stabilized configuration is analyzed \cite{PS13,LPS16,FLP10} — through the dynamics of the sandpile Markov chain, including its mixing behaviour and stationary distribution \cite{JLP19,HJL19,JL07}, to the infinite-volume limit of the stationary distribution and the size of avalanches at criticality \cite{JW14,AJ04,BHJ17,ENP23}.

The set of recurrent configurations for the sandpile Markov chain forms an abelian group — also known as the sandpile group or critical group — which is the focus of this note. In particular, we study the identity element of the sandpile group (shortly \emph{sandpile identity}) on Sierpiński gasket approximation graphs. The sandpile identity  has been examined on various graphs; for example, \cite{LBR02} analyzes it on rectangular grids under certain size constraints, and \cite{M22} investigates the identity on elliptical subsets of the square lattice. Notably, the conjecture of Dhar et al. \cite{DRSV95} concerning a central square of vertices with two particles each in the identity on square boxes in $\mathbb{Z}^2$ remains open.

On the Sierpiński gasket, recent work has characterized and analyzed the identity element of the sandpile group. In particular, \cite{CKF20} describes its structure on Sierpiński gasket graphs under various sinked boundary conditions, where one or more corner vertices are the designated sink vertices and \cite{KSHW24} gives a recursive characterization of the identity on the approximating graphs for normal boundary conditions, where each of the three corner vertices is connected by two edges to the sink vertex. In the current paper, we focus on the scaling limit of the sandpile identity on the approximation graphs of the Sierpiński gasket, previously studied in \cite{KSH26}, where it is shown that the piecewise-constant extensions of the identities converge in the weak* sense to a constant function on the gasket. However, the structure of the identity is lost under the weak* topology. Here, we propose a different notion of limit by convolving the  sandpile identity with the stopped Green's function. This approach retains the identity’s structure and addresses a question raised  in \cite{KSH26}. We now present our main result. Formal definitions of the objects in the statement are deferred to  Section \ref{sec:prelim}.

\begin{theorem}\label{thm:main}
    Let $\mathsf{id}_n$ be the identity of the sandpile group of the $n$-th approximation graph $SG_n$ of the Sierpi\'nski gasket $SG$, for $n\in \mathbb{N}$. Denote by $g_n$ the Green function of the simple random walk on $SG_n$ stopped at the corners, and let $G$ be the Green function with Dirichlet boundary conditions on the corner vertices on the Sierpi\'nski gasket $SG$.
\begin{enumerate}[label=(I\arabic*)]
\item The function $g_n*\mathsf{id}_n$ grows asymptotically as $5^n$, and for all $x\in SG$ we have
        \begin{align*}
            \lim_{n\rightarrow\infty}\frac{1}{5^n}(g_n*\mathsf{id}_n)^{\blacktriangle}(x)=4\int_{SG}G(x,y)d\mu(y).
        \end{align*} \label{I1}

\item The second order term of $g_n*\mathsf{id}_n$ is given by the graph distance to the corners of $SG_n$
        $$\lim_{n\rightarrow\infty}\frac{1}{2^n}\biggl(g_n*\mathsf{id}_n-\frac{8}{3}\sum_{y\in V_{SG_n}}g_n(y,\cdot)\biggr)^{\blacktriangle}(x)=-\frac{1}{3}\mathsf{d}(x),$$
        where $\mathsf{d}(x)$ is the shortest path distance from $x$ to the three corners in $SG$.\label{I2} 

\item For all post-critically finite points $x\in V_*$ it holds that
        $$\lim_{n\rightarrow\infty}\frac{1}{2^n}\biggl(g_n*\mathsf{id}_n(x)-4\cdot5^n\int_{SG}G(x,y)d\mu(y)\biggr)=-\frac{1}{3}\mathsf{d}(x).$$\label{I3}
\end{enumerate}
\end{theorem}
Before proving Theorem \ref{thm:main}, we give an informal and intuitive description of the different scaling factors and objects appearing in the statement:
\begin{itemize}
 \item The factor $5^n$ represents the time acceleration factor required for random walks on $SG_n$ to converge to a continuous Brownian motion on the Sierpinski gasket $SG$.
 \item The constant $8/3$ corresponds to the asymptotic average chip density of the sandpile identity element $\mathsf{id}_n$ as $n \rightarrow \infty$; see Figure \ref{fig:identity_gasket}.
 \item The scaling constant $4$ emerges naturally from the standard convention used to relate the discrete and continuum Laplacians (and their inverse Green's functions) on $SG$.
 \item The measure $\mu$ denotes the standard self-similar probability measure on $SG$.
\end{itemize}
The proof of Theorem \ref{thm:main} relies on decomposing the sandpile identity into the distance to the nearest corner vertex and a function with constant Laplacian in the interior of the gasket. More precisely, after convolving the sandpile identity with the stopped Green function on $SG_n$, we obtain
$g_n * \mathsf{id}_n = \frac{8}{3}\, h_n - \frac{1}{3}\, d_n$,
where $d_n$ denotes the distance to the closest corner, and $h_n$ is the function with Laplacian equal to $1$ at every non-corner vertex of $SG_n$ and value $0$ at the corners. Using results from analysis on fractals, we then deduce the limits stated in Theorem \ref{thm:main}.

\textbf{Outline.} Section \ref{sec:prelim} introduces the abelian sandpile model, the Sierpiński gasket $SG$ and its approximation graphs $SG_n$, and the necessary tools from analysis on fractals used in the proof of Theorem \ref{thm:main}. In Section \ref{sec:decomp}, we establish the key ingredient: a decomposition of the sandpile identity on the approximation graphs into a function with constant Laplacian and the distance to the nearest corner. Finally, in Section \ref{sec:scaling}, we prove Theorem \ref{thm:main}: Proposition \ref{prop:I1} is \ref{I1}, Proposition \ref{prop:I2} is \ref{I2}, and Proposition \ref{prop:I3} is \ref{I3}.

\section{Preliminaries}\label{sec:prelim}

\textbf{Abelian Sandpiles.} For a comprehensive overview on the abelian sandpile model, see the survey \cite{J18}. Let $G = (V \cup \{s\}, E)$ be a finite, connected, undirected graph with a distinguished vertex $s$ (the sink). For $v \in V \cup \{s\}$, let ${\rm deg}_G(v)$ denote its degree in $G$, i.e., the number of its neighbors; when the graph is clear from context, we write ${\rm deg}(v)$. For vertices $v, w \in V$, we write $v \sim w$ if they are adjacent in $G$ (connected by an edge).

A sandpile on $G$ is a function $\sigma: V \to \mathbb{N}$ that assigns to each vertex $v \in V$ the number of particles at $v$. A vertex $v$ is unstable if $\sigma(v) \ge {\rm deg}(v)$; otherwise it is stable. The configuration $\sigma$ is called stable if it is stable at every vertex. For each vertex $v \in V$, define a toppling $T_v$ at $v$ by removing ${\rm deg}(v)$ particles from $v$ and sending one particle along each incident edge. Writing $a_{vw}$ for the number of edges between $v$ and $w$, the toppling redistributes particles as follows:
$$
(T_v \sigma)(u) =
\begin{cases}
\sigma(u) - {\rm deg}(v), & u = v, \\
\sigma(u) + a_{vu}, & u \sim v, \\
\sigma(u), & \text{otherwise}.
\end{cases}
$$
A toppling is legal if it results in a configuration with all nonnegative entries, i.e.~if the toppled vertex $v$ is unstable in $\sigma$. Any particles sent to the sink are permanently removed from the system. We view $\sigma$ as a vector indexed by $V$, with each entry recording the number of particles at the corresponding vertex. With this identification, toppling operations can be expressed using the graph Laplacian of $G$. The Laplacian $\Delta \in {\rm Mat}_{|V \cup \{s\}| \times |V \cup \{s\}|}(\mathbb{Z})$ is the matrix with entries
$$
\Delta_{vw} =
\begin{cases}
- a_{vw}, & v \ne w,\ v \sim w, \\
{\rm deg}(v), & v = w, \\
0, & \text{otherwise}.
\end{cases}
$$
Let $\Delta'$ denote the reduced Laplacian of $G$ with respect to $s$, obtained by restricting $\Delta$ to $V \times V$ (i.e., deleting the row and column corresponding to the sink). Then a toppling at $v$ is given by
$$
T_v \sigma = \sigma - \Delta' \delta_v,
$$
where $\delta_v$ is the vector with a $1$ at index $v$ and $0$ elsewhere. This formulation immediately shows that topplings commute (the model is abelian):
$$
T_v T_w \sigma = T_w T_v \sigma \quad \text{for all } v,w \in V.
$$
The stabilization $\sigma^\circ$ of a sandpile $\sigma$ is the unique stable configuration obtained by performing all legal topplings. This is well defined: there exists a sequence of vertices $v_1,\dots,v_n$ — unique up to permutation — such that all topplings are legal and the resulting configuration $\sigma^\circ = T_{v_n}\cdots T_{v_1}\sigma$ is stable. See \cite{J18} for details.

\textbf{Sandpile Markov chain and the sandpile group.} Given sandpiles and stabilization, one defines a Markov chain on the set of stable configurations of $G$, known as the abelian sandpile Markov chain, as follows.
Let $(X_i)_{i \in \mathbb{N}}$ be i.i.d. uniformly distributed on $V$, and start from any stable configuration $\sigma_0$. Define the Markov chain $(\sigma_n)_{n \in \mathbb{N}}$ by
$$
\sigma_{n+1} = (\sigma_n + \delta_{X_n})^\circ,
$$
i.e., at time-step $n$, given $\sigma_n$, we add one particle to a uniformly chosen vertex in $G$ and then we stabilize $\sigma_n + \delta_{X_n}$. The chain $(\sigma_n)_{n \in \mathbb{N}}$ is called \emph{the abelian sandpile Markov chain}, and its set of recurrent states forms an abelian group denoted $\mathcal{R}_G$, where the group operation
$\oplus: \mathcal{R}_G \times \mathcal{R}_G \rightarrow \mathcal{R}_G$ is the pointwise addition followed by stabilization  $\sigma \oplus \omega = (\sigma + \omega)^\circ$. Since $\mathcal{R}_G$ is an abelian group, it has a neutral element under $\oplus$: a recurrent configuration $\mathsf{id}_G$, called the sandpile identity of $G$.

\begin{figure}
    \centering
    \begin{tikzpicture}[scale=2.5]
        \coordinate (A) at (0,0);
        \coordinate (B) at (1,0);
        \coordinate (C) at (0.5,{sqrt(3)/2});
        
        \draw (A)--(B)--(C)--cycle;
        \foreach \p in {A,B,C}{
          \draw[fill=white] (\p) circle (0.05);
        }
                
    \begin{scope}[shift={(1.6,0)}]
        \coordinate (A) at (0,0);
        \coordinate (B) at (1,0);
        \coordinate (C) at (0.5,{sqrt(3)/2});
        \coordinate (AB) at ($ (A)!0.5!(B) $);
        \coordinate (BC) at ($ (B)!0.5!(C) $);
        \coordinate (CA) at ($ (C)!0.5!(A) $);
        
        \draw (A) -- (B) -- (C) -- cycle;
        \draw (AB) -- (BC) -- (CA) -- cycle;
        
        \foreach \p in {A,B,C,AB,BC,CA}{
          \draw[fill=white] (\p) circle (0.05);
        }
    \end{scope}

    \begin{scope}[shift={(3.2,0)}]
        \coordinate (A) at (0,0);
        \coordinate (B) at (1,0);
        \coordinate (C) at (0.5,{sqrt(3)/2});
        \coordinate (D) at ($ (A)!0.5!(B) $);
        \coordinate (E) at ($ (B)!0.5!(C) $);
        \coordinate (F) at ($ (C)!0.5!(A) $);

        \coordinate (AD) at ($ (A)!0.5!(D) $);
        \coordinate (AE) at ($ (A)!0.5!(E) $);
        \coordinate (AF) at ($ (A)!0.5!(F) $);

        \coordinate (BD) at ($ (B)!0.5!(D) $);
        \coordinate (BE) at ($ (B)!0.5!(E) $);
        \coordinate (BF) at ($ (B)!0.5!(F) $);

        \coordinate (CD) at ($ (C)!0.5!(D) $);
        \coordinate (CE) at ($ (C)!0.5!(E) $);
        \coordinate (CF) at ($ (C)!0.5!(F) $);

        \draw (A) -- (B) -- (C) -- cycle;
        \draw (D) -- (E) -- (F) -- cycle;
        \draw (AD) -- (AE) -- (AF) -- cycle;
        \draw (BD) -- (BE) -- (BF) -- cycle;
        \draw (CD) -- (CE) -- (CF) -- cycle;
        
        \foreach \p in {A,B,C,D,E,F, AD, AE, AF, BD, BE, BF, CD, CE, CF}{
          \draw[fill=white] (\p) circle (0.05);
        }
    \end{scope}
    \end{tikzpicture}
    \caption{The first three Sierpi\'nski gasket approximation graphs $SG_0$, $SG_1$ and $SG_2$.}
    \label{fig:sg_iterations}
\end{figure}
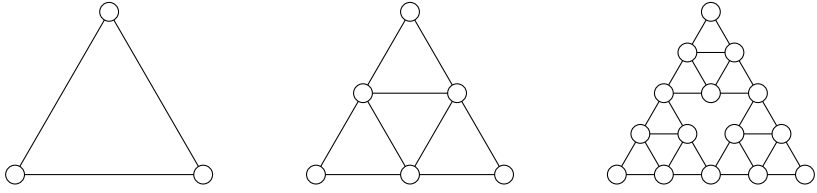

\textbf{The Sierpi\'nski gasket.} We define Sierpi\'nski gasket approximation graphs $SG_n$, for $n\in \N$ as in \cite{BP88}. For $x \in \R^2$ and a graph $G=(V,E)$ that can be embedded in $\R^2$, we denote by $x+G$ the graph \[x+G = (\{x+v : v \in V\}, \, \{(x+a, x+b) : (a,b) \in E\}).\]
For $a\in\R$, we write $aG$ for the graph
$$aG:=(\{av:v\in V\},\{(av,aw):(v,w)\in E\}).$$
Let $u_1=(0,0)$, $u_2 = \frac{1}{2} (1, \sqrt{3})$ and $u_3=(1,0)$. We define $SG_0=(V_{SG_0}, \, E_{SG_0})$ to be the complete graph with vertex and edge set given by
 \[
 V_{SG_0}=\{u_1, u_2, u_3\} \quad \text{and} \quad  
 E_{SG_0}=\{\{u_1,u_2\},\{u_2,u_3\},\{u_1,u_3\}\}\,.
 \]
Recursively, for $n \geq 1$ the $n$-th Sierpi\'nski gasket approximation graph $SG_n=(V_{SG_n}, E_{SG_n})$ is 
\[SG_n = \frac{1}{2} \bigl(SG_{n-1} \cup (u_2 +SG_{n-1}) \cup (u_3 + SG_{n-1})\bigr).\] 
Throughout this paper, we will also use the less cumbersome notation $V_n$ for the set of vertices in the $n$-th iteration. Notice that $SG_n$ is obtained from $SG_{n-1}$ by copying $SG_{n-1}$ three times and shifting one copy to the right and one copy diagonally upwards, and then rescaling the union of the three copies. Figure \ref{fig:sg_iterations} displays the first three Sierpi\'nski gasket graphs $SG_0$, $SG_1$ and $SG_2$.
Define the union of the vertices of all the approximation graphs
\[V_\ast = \bigcup_{n \in \N} V_{n}.\]
\textit{The Sierpi\'nski gasket} is defined as the closure of $V_\ast$ in $\R^2$.

\subsection{Potential Theory on the Sierpi\'nski Gasket}

This subsection is devoted to introducing the necessary tools and framework from analysis on fractals - specifically analysis on the Sierpi\'nski gasket $SG$ - needed to prove Theorem \ref{thm:main}. For a comprehensive overview, see \cite{K01} and \cite{S06}.

\textbf{Standard self-similar measure.} For each $i \in \{1,2,3\}$, define $\psi_i: \mathbb{R}^2 \to \mathbb{R}^2$ by
$$\psi_i(x) = \tfrac{1}{2}(x - u_i) + u_i,$$
where $u_1, u_2, u_3$ are the vertices of the previously defined equilateral triangle. For a word $w=w_1w_2\ldots w_n\in\{1,2,3\}^n$ of length $n\in\N$, we write
$\psi_w=\psi_{w_n}\circ ... \circ \psi_{w_1}$. We take the $\sigma$-algebra of measurable sets on $SG$ to be generated by sets of the form
$\bigl\{\psi_w(SG) : w \in \bigcup_{m \ge 1} \{1,2,3\}^m\bigr\}$.
The standard self-similar measure $\mu$ is the unique measure with $\mu(\psi_w(SG))=3^{-|w|}$,
where $|w|$ is the length of the word $w$. Notice that this standard self-similar measure is up to scalar multiples a $\log_2(3)$-dimensional Hausdorff measure on the Sierpi\'nski gasket with respect to the Euclidean metric. See also Chapter 1.2 in \cite[pp. 5-9]{S06} for more details.

\textbf{Green function on $SG_n$.} Let $(X_t^n)_{t \in \mathbb{N}}$ denote the simple random walk on $SG_n$, i.e., the Markov chain with transition probabilities given by 
\begin{equation*}
p(x,y)=
    \begin{cases}
        \frac{1}{4}, \quad \text{if } x \underset{n}{\sim} y,\, x \neq u_i\\
        \frac{1}{2}, \quad \text{if } x \underset{n}{\sim} y,\, x = u_i
        \end{cases}
\end{equation*}for $i \in \{1,2,3\}$, where $x\underset{n}{\sim}y$ means that $x,y\in V_{SG_n}$ are neighbours in $SG_n$. For $x\in V_{SG_n}$, denote by $\mathbb{E}_x$ the expectation with respect to the probability measure $\mathbb{P}(\cdot|X^n_0=x)$. 
The Green function $g_n(x,y)$ on $SG_n$ is the expected number of visits to $y$ by the walk started at $x$ before it hits one of the three corner vertices $u_1,u_2,u_3$:
$$
g_n(x,y) = \mathbb{E}_x\!\biggl[\sum_{t=0}^{\tau_n-1} \mathbf{1}_{\{X_t^n = y\}}\biggr],
\quad \text{where } \tau_n = \inf\bigl\{t \ge 0 : X_t^n \in \{u_1,u_2,u_3\}\bigr\}.
$$
If we write $\Delta_n$ for the graph Laplacian on $SG_n$, then it holds that
$$\Delta_n g_n(x,\cdot)(y)=\delta_x(y), \quad \text{for all} \quad 
x,y\in SG_n\backslash\{u_1,u_2,u_3\}.$$

\textbf{Piecewise constant continuation to $SG$.} For $x \in SG$, define $x_n \in V_{SG_n}$ as $x_n = {\rm argmin}\{d(x,y): y \in SG_n\}$. For uniqueness, if there are several minimizers, choose the one furthest to the right. For a  function
$f:V_{SG_n} \rightarrow \R$ we define the piecewise constant continuation of $f$ to $SG$ by $f^\blacktriangle : SG \rightarrow \R, f^\blacktriangle(x)= f(x_n)$.

\textbf{Laplacian operator on $SG$.} The Laplacian operator on the Sierpi\'nski gasket $SG$, denoted by $\Delta$, can be defined via a weak formulation using an energy form on $SG$. For $n\in \N$, we define the energy form $\mathcal{E}_n$ on $SG_n$ as
$$\mathcal{E}_n(f,g)=\bigg(\frac{5}{3}\bigg)^n\sum_{x\underset{n}{\sim} y}\bigl(f(x)-f(y)\bigr)\bigl(g(x)-g(y)\bigr),$$
where $f,g:SG\rightarrow\R$. Observe that the $5/3$ factor in the energy functional corresponds to the increase in minimal energy when a function on the $n$-th gasket is extended to the $(n+1)$-st. For details, see \cite[Section 1.3.]{S06}.
Using these finite iteration energy forms $\mathcal{E}_n$ we define an energy form on $SG$ via
$$\mathcal{E}(f,g)=\lim_{n\rightarrow\infty}\mathcal{E}_n(f|_{V_{SG_n}},g|_{V_{SG_n}}).$$
Note that the limit is not necessarily finite. So, we also define the domain of the energy functional as
$$\operatorname{dom}(\mathcal{E}) := \{ f : SG \to \mathbb{R} \mid \mathcal{E}(f,f) < \infty \}.$$
With this energy form, we define a weak Laplacian on $SG$. For $u \in \operatorname{dom}(\mathcal{E})$ and continuous $f: SG \to \mathbb{R}$, we write $u \in \operatorname{dom}(\Delta)$ and $\Delta u = f$ if and only if
$$
\mathcal{E}(u, v) = - \int_{SG} f\, v \, d\mu \quad \text{for all } v \in \operatorname{dom}(\mathcal{E}).
$$
\textbf{Green function on $SG$.} 
For Theorem \ref{thm:main}, we must specify the function on $SG$ that arises as the scaling limit of the discrete Green functions $g_n$ defined above. Specifically, we define the Green function on the Sierpi\'nski gasket $SG$, i.e., the integral kernel of the inverse Laplacian on SG.
We define the Green function subject to Dirichlet boundary conditions on $SG$ via harmonic splines as explained in \cite[Section 2.6.]{S06} by
$$
G(x,y) := \lim_{M \to \infty} \sum_{m=0}^{M} \sum_{z,z' \in V_{SG_{m+1}} \setminus V_{SG_m}} g(z,z')\, \Psi_{z}^{m+1}(x)\, \Psi_{z'}^{m+1}(y),
$$
where
$$
g(z,z') =
\begin{cases}
\dfrac{9}{50}\Big(\dfrac{3}{5}\Big)^m, & z = z' \ \text{and}\ z \in V_{SG_{m+1}} \setminus V_{SG_m},\\[6pt]
\dfrac{3}{50}\Big(\dfrac{3}{5}\Big)^m, & z \ne z',\ z,z' \in V_{SG_{m+1}} \setminus V_{SG_m},\ \text{and}\ z,z' \in \psi_w(SG)\ \text{for some } |w|=m,
\end{cases}
$$
and, for $z \in V_{SG_m}$, the level-$m$ harmonic spline $\Psi_z^m$ is the unique piecewise harmonic function on $SG$ satisfying
$$
\Psi_z^m(x) =
\begin{cases}
1, & x = z,\\
0, & x \in V_{m} \setminus \{z\}.
\end{cases}
$$
A key feature of the Green function on $SG$ is that, up to the scaling factor $(3/5)^n$, it matches the stopped Green function $g_n$ on the vertices of $SG_n$.

\begin{lemma}\label{lem:green-is-same}
For all $n \in \mathbb{N}$ and all $x,y \in V_{n}$ it holds
$\left(\frac{3}{5}\right)^n g_n(x,y) = G(x,y).$
\end{lemma}
For a proof see, for instance, Lemma 2.4 in \cite{FHKS24}.

\subsubsection{Auxiliary results}
In the proofs of the scaling limit in Section \ref{sec:scaling}, we need several technical lemmas. These include the mean-value property of harmonic functions on $SG$ and properties of the Green function $G$ on the $SG$, specifically the variation of the Green function within a given $n$-cell.

\begin{lemma}\label{lem:mean-value}
    Let $h:SG\rightarrow\R$ be a harmonic function on $SG\backslash V_0$. Then for every $m,n\in\N$ with $m>n$ and $x\in V_n\backslash V_0$ it holds that
    $$\frac{1}{\mu(B(x,2^{-m}))}\int_{B(x,2^{-m})}h(y)d\mu(y)=h(x).$$
\end{lemma}
\begin{proof}
The statement follows from the well-known fact \cite[Section 1.3.]{S06} that a harmonic function $h$ on an n-cell of $SG$ is determined by its values at the three corner vertices via the $1/5–2/5$ rule.
More precisely, if $x_{1},x_{2},x_{3} \in V_{n}$ are the corner vertices and$y_{1},y_{2},y_{3}\in V_{n+1}\backslash V_{n}$ are the corresponding cut points (each $y_i$ opposite $x_i$), the rule gives $h=h(x_{1})\psi_{1}+h(x_{2})\psi_{2}+h(x_{3})\psi_{3}$
on this $n$-cell, where $\psi_{i}$ is the unique harmonic spline with $\psi_{i}(x_{j})=\delta_{ij}$, and the algorithm gives
\begin{align}\label{eq:splines}
\psi_{1}(y_{1})=\frac{1}{5},\quad \psi_{1}(y_{2})=\psi_{1}(y_{3})=\frac{2}{5},
\end{align}
The other values of $\psi_{i}(y_{j})$ are found via dihedral symmetry. Consequently $h$ is discrete harmonic:
\begin{align}\label{eq:h-repres}
h(y_{1})=\frac{1}{4}(h(x_{2})+h(x_{3})+h(y_{2})+h(y_{3})).
\end{align}
Moreover, since $\sum_{i=1}^{3}\psi_{i}\equiv1$ and $\mu(m\text{-cell})=3^{-m}$ for every $m\ge n$ we deduce by symmetry that for every $(n+1)$-subcell $\mathcal{T}_{n+1}$ contained in the $n$-cell, $\int_{\mathcal{T}_{n+1}}\psi_{i}d\mu=3^{-(n+2)}$ for $i\in\{1,2,3\}$.
We now establish the mean-value property of $h$ over the ball $B(y_{1},2^{-(n+1)})$:
\begin{align}\label{eq:mean-value}
\frac{1}{\mu(B(y_{1},2^{-(n+1)}))}\int_{B(y_{1},2^{-(n+1)})}h~d\mu=h(y_{1}).
\end{align}
Using (\ref{eq:splines}) and the above integrals of $\psi_{i}$, we get
\begin{align*}
\int_{B(y_{1},2^{-(n+1)})}h~d\mu&=\big(h(y_{1})+h(x_{3})+h(y_{2})\big)\cdot3^{-(n+2)}+\big(h(y_{1})+h(x_{2})+h(y_{3})\big)\cdot3^{-(n+2)} \\
&=\frac{1}{3}\cdot3^{-(n+1)}\cdot\big(h(x_{2})+h(x_{3})+h(y_{2})+h(y_{3})+2h(y_{1})\big) \\
&=\frac{1}{3}\cdot3^{-(n+1)}\cdot6h(y_{1}),
\end{align*}
where we used \eqref{eq:h-repres} in the last equation. Dividing the last expression by $\mu(B(y_{1},2^{-(n+1)}))=2\cdot3^{-(n+1)}$ gives \eqref{eq:mean-value}. In \eqref{eq:mean-value} we apply the self-similar cell structure of $SG$ to deduce the claimed mean-value property. 
\end{proof}
The next result relates the Green function on $SG$ to the effective resistance. The definition of the effective resistance metric is standard, and can for example be found in \cite[Section 1.6.]{S06}.

\begin{proposition}\label{prop:effective-res}
    Let $R_{eff}(A,B)$ be the effective resistance between subsets $A, B\subset SG$, subject to Dirichlet boundary conditions on $V_{0}$. Let $G(x,y)$ be the Green's function on SG subject to Dirichlet boundary conditions on $V_{0}$. Then the following holds:
\begin{enumerate}
\setlength\itemsep{0em}
    \item[(1)] For every $x\in SG$, $G(x,x)=R_{eff}(x,V_{0})$.
    \item[(2)] For every $x,y\in SG$, $G(x,y)=\frac{1}{2}(G(x,x)+G(y,y)-R_{eff}(x,y))$.
    \item[(3)] For every $x,y\in SG$ we have $G(x,x)\ge G(x,y)$.
\end{enumerate}
\end{proposition} 
\begin{proof}
These follow from the general theory of Green functions on resistance forms, which is detailed in \cite[Section 4]{Kig12}. For proofs, see Theorems 4.1 and Theorem 4.3 from \cite{Kig12}.
\end{proof}

\begin{lemma}\label{lem:green-variation}
For every $x,y,z\in SG$ we have
$$|G(x,y)-G(x,z)|\le R_{eff}(y,z),$$
and consequently
$$0\le G(x,x)-G(x,y)\le R_{eff}(x,y).$$
\end{lemma}

\begin{proof}
Using $(1)$ and $(2)$ in Proposition \ref{prop:effective-res}, we obtain
\begin{align*}
G(x,y)-G(x,z)&=\frac{1}{2}[G(y,y)-G(z,z)+R_{eff}(x,z)-R_{eff}(x,y)] \\
&=\frac{1}{2}[R_{eff}(y,V_{0})-R_{eff}(z,V_{0})+R_{eff}(x,z)-R_{eff}(x,y)].
\end{align*}
Note that $V_{0}$ is grounded and hence can be collapsed to a single vertex. Meanwhile, the resistance metric satisfies the triangle inequality, so
\begin{align*}
|G(x,y)-G(x,z)|&\le\frac{1}{2}(|R_{eff}(y,V_{0})-R_{eff}(z,V_{0})|+|R_{eff}(x,z)-R_{eff}(x,y)|) \\
&\le\frac{1}{2}(R_{eff}(y,z)+R_{eff}(y,z))=R_{eff}(y,z),
\end{align*}
proving the first inequality. This together with $(3)$ from Proposition 3 implies the second part.
\end{proof}

\begin{lemma}\label{lem:green-bounds}
There exists a constant $C>0$ such that the following hold.
\begin{enumerate}\setlength\itemsep{0em}
    \item[(1)] For every $x,y\in SG$ belonging to the same $n$-cell, we have
    $$0\le G(x,x)-G(x,y)\le C\Big(\frac{3}{5}\Big)^{n},$$
    \item[(2)] For every $x\in SG$, $n\in\mathbb{N}$ and $y\in SG$ with $d(y,V_{0})\le2^{-n}$ we have
    $$0\le G(x,y)\le C\Big(\frac{3}{5}\Big)^{n}.$$
\end{enumerate}
\end{lemma}
\begin{proof}
For the first claim, using Lemma \ref{lem:green-variation} it suffices to show that there exists $C>0$ such that for every $x,y\in SG$ belonging to the same $n$-cell, we have $R_{eff}(x,y)\le C\left(\frac{3}{5}\right)^{n}$. This is the well-known resistance upper bound \cite[Section 1.6]{S06}.
For the second claim, using the third and then the first item in Proposition \ref{prop:effective-res}, we obtain $G(x,y)\le G(y,y)=R_{eff}(y,V_{0})$. Since $y$ shares the same $n$-cell with one of the corner vertices in $V_{0}$, applying the aforementioned resistance upper bound yields $R_{eff}(y,V_{0})\le C\left(\frac{3}{5}\right)^{n}$.
\end{proof}
\subsection{Sandpiles on the Sierpi\'nski gasket graphs}\label{subsection:sand}

\textbf{Boundary conditions  on $SG_n$.} 
For the abelian sandpile Markov chain on $SG_n$, we adopt the standard “normal” boundary conditions used in the literature (e.g., [DV98]): introduce a sink vertex $s \notin V_{SG_n}$ and connect it by two edges to the corner vertices of $SG_n$, so every vertex in $V_{SG_n}$ has degree $4$ (see Figure $\ref{fig:sg_normalbdry}$). Throughout, “sandpile group on $SG_n$” refers to the set of recurrent configurations of the sandpile Markov chain with this boundary, denoted $\mathcal{R}_n$, and we write $\mathsf{id}_n$ for its identity element.

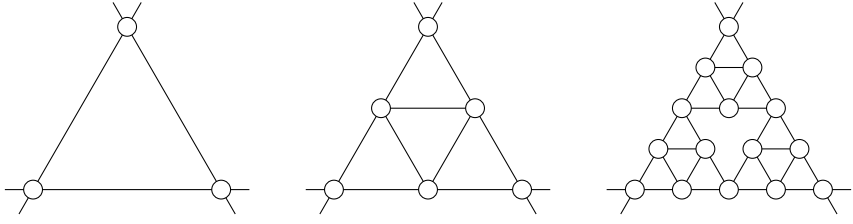
\begin{figure}
    \centering
    \begin{tikzpicture}[scale=2.5]
        \coordinate (A) at (0,0);
        \coordinate (B) at (1,0);
        \coordinate (C) at (0.5,{sqrt(3)/2});
        
        \draw (A)--(B)--(C)--cycle;
        \foreach \p/\a/\b in {
              A/180/240,
              B/0/-60,
              C/60/120
            }{
              \draw (\p) -- ++(\a:0.15);
              \draw (\p) -- ++(\b:0.15);
            }
        \foreach \p in {A,B,C}{
          \draw[fill=white] (\p) circle (0.05);
        }
                
    \begin{scope}[shift={(1.6,0)}]
        \coordinate (A) at (0,0);
        \coordinate (B) at (1,0);
        \coordinate (C) at (0.5,{sqrt(3)/2});
        \coordinate (AB) at ($ (A)!0.5!(B) $);
        \coordinate (BC) at ($ (B)!0.5!(C) $);
        \coordinate (CA) at ($ (C)!0.5!(A) $);
        
        \draw (A) -- (B) -- (C) -- cycle;
        \draw (AB) -- (BC) -- (CA) -- cycle;
        \foreach \p/\a/\b in {
          A/180/240,
          B/0/-60,
          C/60/120
        }{
          \draw (\p) -- ++(\a:0.15);
          \draw (\p) -- ++(\b:0.15);
        }
        \foreach \p in {A,B,C,AB,BC,CA}{
          \draw[fill=white] (\p) circle (0.05);
        }
    \end{scope}

    \begin{scope}[shift={(3.2,0)}]
        \coordinate (A) at (0,0);
        \coordinate (B) at (1,0);
        \coordinate (C) at (0.5,{sqrt(3)/2});
        \coordinate (D) at ($ (A)!0.5!(B) $);
        \coordinate (E) at ($ (B)!0.5!(C) $);
        \coordinate (F) at ($ (C)!0.5!(A) $);

        \coordinate (AD) at ($ (A)!0.5!(D) $);
        \coordinate (AE) at ($ (A)!0.5!(E) $);
        \coordinate (AF) at ($ (A)!0.5!(F) $);

        \coordinate (BD) at ($ (B)!0.5!(D) $);
        \coordinate (BE) at ($ (B)!0.5!(E) $);
        \coordinate (BF) at ($ (B)!0.5!(F) $);

        \coordinate (CD) at ($ (C)!0.5!(D) $);
        \coordinate (CE) at ($ (C)!0.5!(E) $);
        \coordinate (CF) at ($ (C)!0.5!(F) $);

        \draw (A) -- (B) -- (C) -- cycle;
        \draw (D) -- (E) -- (F) -- cycle;
        \draw (AD) -- (AE) -- (AF) -- cycle;
        \draw (BD) -- (BE) -- (BF) -- cycle;
        \draw (CD) -- (CE) -- (CF) -- cycle;

        \foreach \p/\a/\b in {
              A/180/240,
              B/0/-60,
              C/60/120
            }{
              \draw (\p) -- ++(\a:0.15);
              \draw (\p) -- ++(\b:0.15);
            }
        \foreach \p in {A,B,C,D,E,F, AD, AE, AF, BD, BE, BF, CD, CE, CF}{
          \draw[fill=white] (\p) circle (0.05);
        }
    \end{scope}
    \end{tikzpicture}
    \caption{The Sierpi\'nski gasket graphs $SG_0$, $SG_1$ and $SG_2$ with normal boundary conditions.}
    \label{fig:sg_normalbdry}
\end{figure}

\textbf{Sandpile identity on $SG_n$.} Define the sandpile configuration $M_1(x,y,z)$ on $SG_1$ as
\begin{center}
    \begin{tikzpicture}[scale=2.5, baseline=-0.5ex]
        \coordinate (A) at (0,0);
        \coordinate (B) at (1,0);
        \coordinate (C) at (0.5,{sqrt(3)/2});
        \coordinate (AB) at ($ (A)!0.5!(B) $);
        \coordinate (BC) at ($ (B)!0.5!(C) $);
        \coordinate (CA) at ($ (C)!0.5!(A) $);
        
        \draw (A) -- (B) -- (C) -- cycle;
        \draw (AB) -- (BC) -- (CA) -- cycle;
        
        \foreach \p in {A,B,C,AB,BC,CA}{
          \draw[fill=white] (\p) circle (0.1);
        }
        \node at (A) {$x$};
        \node at (B) {$y$};
        \node at (C) {$z$};
        \node at (AB) {$3$};
        \node at (BC) {$2$};
        \node at (CA) {$3$};
        
        \node at (1.3,0.4) {$=$};
    
    \begin{scope}[shift={(1.6,0)}]
        \coordinate (A2) at (0,0);
        \coordinate (B2) at (1,0);
        \coordinate (C2) at (0.5,{sqrt(3)/2});
        \draw (A2)--(B2)--(C2)--cycle;
        
        \foreach \p/\t in {A2/x, B2/y, C2/z}{
        \draw[fill=white] (\p) circle (0.1);
        \node at (\p) {$\t$};
        }
        
        \node at (0.5,0.4) {$M_1$};
    \end{scope}

        \node[right] at (2.8,0.4) {$=\quad M_1(x,y,z)$};
    \end{tikzpicture}
\end{center}
where the corner values $x,y,z \in \N$ are arbitrary. Iteratively define the configuration $M_{n+1}(x,y,z)$ 

\begin{center}
    \begin{tikzpicture}[scale=2.5, baseline=-0.5ex]
        \coordinate (A) at (0,0);
        \coordinate (B) at (1,0);
        \coordinate (C) at (0.5,{sqrt(3)/2});
        \coordinate (AB) at ($ (A)!0.5!(B) $);
        \coordinate (BC) at ($ (B)!0.5!(C) $);
        \coordinate (CA) at ($ (C)!0.5!(A) $);
        
        \draw (A) -- (B) -- (C) -- cycle;
        \draw (AB) -- (BC) -- (CA) -- cycle;
        
        \foreach \p in {A,B,C,AB,BC,CA}{
          \draw[fill=white] (\p) circle (0.1);
        }
        \node at (A) {$x$};
        \node at (B) {$y$};
        \node at (C) {$z$};
        \node at (AB) {$3$};
        \node at (BC) {$2$};
        \node at (CA) {$3$};
        
        \node at (0.25,0.144) {$M_n$};
        \node at (0.5,0.577) {$M_n$};
        \node at (0.75,0.144) {$M_n$};
        
        \node at (1.3,0.4) {$=$};
    
    \begin{scope}[shift={(1.6,0)}]
        \coordinate (A2) at (0,0);
        \coordinate (B2) at (1,0);
        \coordinate (C2) at (0.5,{sqrt(3)/2});
        \draw (A2)--(B2)--(C2)--cycle;
        
        \foreach \p/\t in {A2/x, B2/y, C2/z}{
        \draw[fill=white] (\p) circle (0.1);
        \node at (\p) {$\t$};
        }
        
        \node at (0.5,0.4) {$M_{n+1}$};
    \end{scope}
    
        \node[right] at (2.8,0.4) {$=\quad M_{n+1}(x,y,z)$.};
    \end{tikzpicture}
\end{center}

In \cite{KSHW24}, the authors prove in Theorem 3.2 that for any $n \geq 1$, the identity element of $SG_{n+1}$ with normal boundary conditions is given by 

\begin{center}
    \begin{tikzpicture}[scale=2.5, baseline=-0.5ex]
        \coordinate (A) at (0,0);
        \coordinate (B) at (1,0);
        \coordinate (C) at (0.5,{sqrt(3)/2});
        \coordinate (AB) at ($ (A)!0.5!(B) $);
        \coordinate (BC) at ($ (B)!0.5!(C) $);
        \coordinate (CA) at ($ (C)!0.5!(A) $);
        
        \draw (A) -- (B) -- (C) -- cycle;
        \draw (AB) -- (BC) -- (CA) -- cycle;
        
        \foreach \p in {A,B,C,AB,BC,CA}{
          \draw[fill=white] (\p) circle (0.1);
        }
        \node at (A) {$2$};
        \node at (B) {$2$};
        \node at (C) {$2$};
        \node at (AB) {$2$};
        \node at (BC) {$2$};
        \node at (CA) {$2$};
        
        \node at (0.25,0.144) {$M_n$};
        \node at (0.5,0.577) {$M_n^-$};
        \node at (0.75,0.144) {$M_n^+$};
        
        \node at (1.3,0.4) {$= \quad \mathsf{id}_{n+1}$,};
    \end{tikzpicture}
\end{center}

where $M_n^+$ (respectively $M_n^-$) denotes the sandpile configuration obtained from $M_n$ by rotating $SG_n$  counterclockwise (respectively clockwise) by $120^\circ$. The sandpile identity elements on the Sierpi\'nski gaskets $SG_n$ for $n \in \{2,3,4,5\}$ are displayed in Figure \ref{fig:identity_gasket}.

\section{Decomposition of the sandpile identity}\label{sec:decomp}

In this section, we establish a decomposition of the sandpile identity needed for proving the main theorem. For $n \in \mathbb{N}$, define $I_n: V_{SG_n} \to \mathbb{R}$ by
$$I_n(x) = (g_n * \mathsf{id}_n)(x) := \sum_{y \in V_{SG_n}} g_n(x,y)\,\mathsf{id}_n(y).$$
We have $\Delta_n I_n = \mathsf{id}_n$ on $V_{SG_n}\setminus\{u_1,u_2,u_3\}$ and $I_n(u_i)=0$ for each $i\in\{1,2,3\}$; that is, $I_n$ vanishes at the corner vertices of $SG_n$. We also define $h_n: V_{SG_n}\to\mathbb{R}$ by
\begin{align}\label{eq:h_n}
h_n(x) = \sum_{y\in V_{SG_n}} g_n(x,y).
\end{align}
It follows that $\Delta_n h_n(x)=1$ for all $x\in V_{SG_n}\setminus\{u_1,u_2,u_3\}$, and $h_n$ vanishes at the corner vertices of $SG_n$. Next, define the distance to the corners by
\begin{align}\label{eq:corner-dist}
d_n: V_{SG_n}\to\mathbb{R},\quad d_n(x)= \min\bigl\{\,d_{SG_n}(x,u_i): i\in\{1,2,3\}\,\bigr\},
\end{align}
where $d_{SG_n}$ denotes the graph distance in $SG_n$. In the next proposition, we show that $I_n$ can be expressed as a linear combination of the functions $d_n$ and $h_n$.

\begin{figure}
    \centering
    \begin{minipage}{0.24\textwidth}
        \centering
        \includegraphics[width=\linewidth]{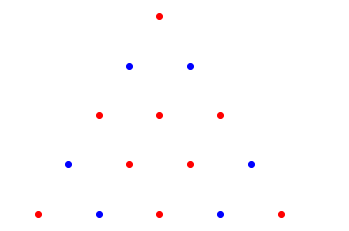}\\
        \small $\mathsf{id}_2$
    \end{minipage}
    \begin{minipage}{0.24\textwidth}
        \centering
        \includegraphics[width=\linewidth]{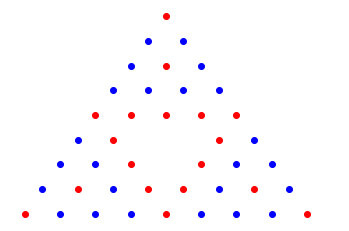}\\
        \small $\mathsf{id}_3$
    \end{minipage}
    \begin{minipage}{0.24\textwidth}
        \centering
        \includegraphics[width=\linewidth]{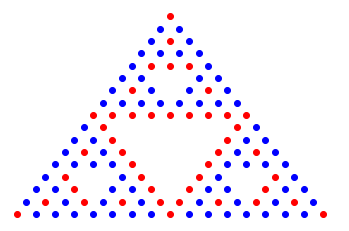}\\
        \small $\mathsf{id}_4$
    \end{minipage}
    \begin{minipage}{0.24\textwidth}
        \centering
        \includegraphics[width=\linewidth]{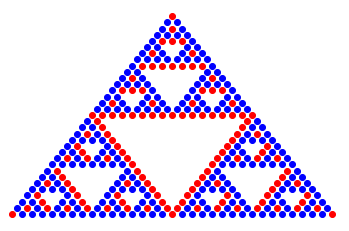}\\
        \small $\mathsf{id}_5$
    \end{minipage}
    \caption{The identity element  $\mathsf{id}_n$ on $SG_n$ for $n \in \{2,3,4,5\}$. Blue dots are vertices with 3 chips and red dots are vertices with 2 chips.}
    \label{fig:identity_gasket}
\end{figure}

\begin{proposition}\label{prop:id-decomp}
For $n \geq 2$ and $I_n=g_n*\mathsf{id}_n$, it holds that
    $$I_n=-\frac{1}{3}d_n+\frac{8}{3}h_n.$$
\end{proposition}
\begin{proof}
We write $F_n:=-\frac{1}{3}d_n+\frac{8}{3}h_n$ and we first show that
$$\Delta_n F_n(x)=\mathsf{id}_n(x) \quad \text{for all} \quad
x\in V_{SG_n}\backslash\{u_1,u_2,u_3\}.$$
Since the functions $d_n$ and $h_n$ are rotationally symmetric, it 
suffices to consider $\Delta_n F_n$ only on the top triangle $\bigtriangleup_n^2$. Special attention must be paid to the cutpoints whose neighborhoods intersect more than one of the three subtriangles, since the Laplacian at a vertex depends on the values of the function on its neighboring vertices. We write
$\triangle_n^2$ for the vertices of $SG_n$ in the upper triangle $\triangle_n^2=\psi_2(V_{SG_{n-1}}\backslash\{u_1,u_2,u_3\})$ and similarly
$$\triangle_n^1=\psi_1\bigl(V_{SG_{n-1}}\backslash\{u_1,u_2,u_3\}\bigr),\quad \triangle_n^3=\psi_3\bigl(V_{SG_{n-1}}\backslash\{u_1,u_2,u_3\}\bigr).$$
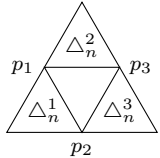
\begin{figure}[h]
\centering
    \begin{tikzpicture}[scale=2, font=\scriptsize]
        \coordinate (A) at (0,0);
        \coordinate (B) at (1,0);
        \coordinate (C) at (0.5,{sqrt(3)/2});
        \coordinate (AB) at ($ (A)!0.5!(B) $);
        \coordinate (BC) at ($ (B)!0.5!(C) $);
        \coordinate (CA) at ($ (C)!0.5!(A) $);
        
        \draw (A) -- (B) -- (C) -- cycle;
        \draw (AB) -- (BC) -- (CA) -- cycle;
    
        \node at (0.25,0.144) {$\bigtriangleup_n^1$};
        \node at (0.5,0.577) {$\bigtriangleup_n^2$};
        \node at (0.75,0.144) {$\bigtriangleup_n^3$};
    
        \node[below] at (AB) {$p_2$};
        \node[right] at (BC) {$p_3$};
        \node[left] at (CA) {$p_1$};
\end{tikzpicture}
\caption{Illustration of the cutpoints $p_1,p_2,p_3$ and the subtriangles $\triangle_n^i$ for $i\in\{1,2,3\}$ in $SG_n$.}
\label{fig:labels-for-proof}
\end{figure}
We write  $p_1,p_2$ and $p_3$ for the three cutpoints, so
$p_1=\frac{1}{2}u_2$, $p_2=\frac{1}{2}u_3$, $p_3=p_1+p_2$; see 
Figure \ref{fig:labels-for-proof}.
Below, for readability reasons, we suppress the subscript $n$ of the Laplacians in the calculations.
We prove the statement by induction on the iteration of the gasket graphs; let $n=2$. 

The graph distance from the top vertex $u_2$ and its Laplacian are given by
    \begin{center}
    \begin{tikzpicture}[scale=1.75, baseline=-0.5ex]
        \node[left] at (-0.3,0.4) {$d_2(x) \vert_{\bigtriangleup_2^2\cup\{p_1,p_3\}}=$};

        \coordinate (A) at (0,0);
        \coordinate (B) at (1,0);
        \coordinate (C) at (0.5,{sqrt(3)/2});
        \coordinate (AB) at ($ (A)!0.5!(B) $);
        \coordinate (BC) at ($ (B)!0.5!(C) $);
        \coordinate (CA) at ($ (C)!0.5!(A) $);
        
        \draw (A) -- (B) -- (C) -- cycle;
        \draw (AB) -- (BC) -- (CA) -- cycle;

        \filldraw[black] (C) circle (0.03);
        
        \node[above] at (C) {};
        \node[right] at (BC) {1};
        \node[left] at (CA) {1};
        \node[below] at (A) {2};
        \node[below] at (B) {2};
        \node[below] at (AB) {2};
    \end{tikzpicture}
\hspace{0.9cm}
    \begin{tikzpicture}[scale=1.75, baseline=-0.5ex]
        \node[left] at (-0.3,0.4) {$\Delta d_2(x) \vert_{\bigtriangleup_2^2}=$};

        \coordinate (A) at (0,0);
        \coordinate (B) at (1,0);
        \coordinate (C) at (0.5,{sqrt(3)/2});
        \coordinate (AB) at ($ (A)!0.5!(B) $);
        \coordinate (BC) at ($ (B)!0.5!(C) $);
        \coordinate (CA) at ($ (C)!0.5!(A) $);
        
        \draw (A) -- (B) -- (C) -- cycle;
        \draw (AB) -- (BC) -- (CA) -- cycle;

        \filldraw[black] (C) circle (0.03);
        
        \node[right] at (BC) {-1};
        \node[left] at (CA) {-1};
        \node[below] at (AB) {2};

        \draw[fill=white] (A) circle (0.03);
        \draw[fill=white] (B) circle (0.03);
    \end{tikzpicture}
    \end{center}
and the white circles placed at the cutpoints indicate values that we compute later. To emphasize that we work on $\bigtriangleup_2^2$, we have colored the top vertex $u_2$ black.
Since the Laplacian of $h_2$ is $1$ in all non-corner vertices, the Laplacian of $F_2$ on $\bigtriangleup_2^2$ can be computed as 

\[
\begin{aligned}
\Delta F_2 \vert_{\bigtriangleup_2^2}
&= -\tfrac{1}{3}\Delta d_2\vert_{\bigtriangleup_2^2}
   + \tfrac{8}{3}\Delta h_2\vert_{\bigtriangleup_2^2} \\
&=
\begin{tikzpicture}[scale=1.75, baseline={(0,0.5)}]
        \coordinate (A) at (0,0);
        \coordinate (B) at (1,0);
        \coordinate (C) at (0.5,{sqrt(3)/2});
        
        \coordinate (AB) at ($ (A)!0.5!(B) $);
        \coordinate (BC) at ($ (B)!0.5!(C) $);
        \coordinate (CA) at ($ (C)!0.5!(A) $);
        
        \draw (A) -- (B) -- (C) -- cycle;
        \draw (AB) -- (BC) -- (CA) -- cycle;

        \filldraw[black] (C) circle (0.03);
        
        \draw[fill=white] (A) circle (0.03);
        \draw[fill=white] (B) circle (0.03);

        \node[above] at (C) {};
        \node[below] at (AB) {$-2/3$};
        \node[right] at (BC) {$1/3$};
        \node[left] at (CA) {$1/3$};
\end{tikzpicture}
\; + \;
\begin{tikzpicture}[scale=1.75, baseline={(0,0.5)}]
        \coordinate (A) at (0,0);
        \coordinate (B) at (1,0);
        \coordinate (C) at (0.5,{sqrt(3)/2});
        
        \coordinate (AB) at ($ (A)!0.5!(B) $);
        \coordinate (BC) at ($ (B)!0.5!(C) $);
        \coordinate (CA) at ($ (C)!0.5!(A) $);
        
        \draw (A) -- (B) -- (C) -- cycle;
        \draw (AB) -- (BC) -- (CA) -- cycle;

        \filldraw[black] (C) circle (0.03);
        
        \draw[fill=white] (A) circle (0.03);
        \draw[fill=white] (B) circle (0.03);

        \node[below] at (AB) {$8/3$};
        \node[right] at (BC) {$8/3$};
        \node[left] at (CA) {$8/3$};
\end{tikzpicture}
\; = \;
\begin{tikzpicture}[scale=1.75, baseline={(0,0.5)}]
        \coordinate (A) at (0,0);
        \coordinate (B) at (1,0);
        \coordinate (C) at (0.5,{sqrt(3)/2});
        
        \coordinate (AB) at ($ (A)!0.5!(B) $);
        \coordinate (BC) at ($ (B)!0.5!(C) $);
        \coordinate (CA) at ($ (C)!0.5!(A) $);
        
        \draw (A) -- (B) -- (C) -- cycle;
        \draw (AB) -- (BC) -- (CA) -- cycle;

        \filldraw[black] (C) circle (0.03);
        
        \draw[fill=white] (A) circle (0.03);
        \draw[fill=white] (B) circle (0.03);
        \node[above] at (C) {};
        \node[below] at (AB) {$2$};
        \node[right] at (BC) {$3$};
        \node[left] at (CA) {$3$};
\end{tikzpicture}
 \, = \quad M_1^-.
\end{aligned}
\]

For the values at the cutpoints we have
\[
\begin{aligned}
\Delta F_2 \vert_{\{p_1,p_2,p_3\}} & = -\frac{1}{3}\Delta d_2\vert_{\{p_1,p_2,p_3\}} + \frac{8}{3}\Delta h_2\vert_{\{p_1,p_2,p_3\}} \\
           & =
\begin{tikzpicture}[scale=2, font=\small, baseline={(0,0.5)}]
        \coordinate (A) at (0,0);
        \coordinate (B) at (1,0);
        \coordinate (C) at (0.5,{sqrt(3)/2});
        \coordinate (AB) at ($ (A)!0.5!(B) $);
        \coordinate (BC) at ($ (B)!0.5!(C) $);
        \coordinate (CA) at ($ (C)!0.5!(A) $);
        
        \draw (A) -- (B) -- (C) -- cycle;
        \draw (AB) -- (BC) -- (CA) -- cycle;
    
        \node at (0.25,0.144) {$\bigtriangleup_2^1$};
        \node at (0.5,0.577) {$\bigtriangleup_2^2$};
        \node at (0.75,0.144) {$\bigtriangleup_2^3$};

        \node[above] at (C) {};
        \node[below] at (AB) {$-2/3$};
        \node[right] at (BC) {$-2/3$};
        \node[left] at (CA) {$-2/3$};
\end{tikzpicture}
\; + \;
\begin{tikzpicture}[scale=2, font=\small, baseline={(0,0.5)}]
        \coordinate (A) at (0,0);
        \coordinate (B) at (1,0);
        \coordinate (C) at (0.5,{sqrt(3)/2});
        \coordinate (AB) at ($ (A)!0.5!(B) $);
        \coordinate (BC) at ($ (B)!0.5!(C) $);
        \coordinate (CA) at ($ (C)!0.5!(A) $);
        
        \draw (A) -- (B) -- (C) -- cycle;
        \draw (AB) -- (BC) -- (CA) -- cycle;
    
        \node at (0.25,0.144) {$\bigtriangleup_2^1$};
        \node at (0.5,0.577) {$\bigtriangleup_2^2$};
        \node at (0.75,0.144) {$\bigtriangleup_2^3$};
    
        \node[below] at (AB) {$8/3$};
        \node[right] at (BC) {$8/3$};
        \node[left] at (CA) {$8/3$};
\end{tikzpicture}
\; = \;
\begin{tikzpicture}[scale=2, font=\small, baseline={(0,0.5)}]
        \coordinate (A) at (0,0);
        \coordinate (B) at (1,0);
        \coordinate (C) at (0.5,{sqrt(3)/2});
        \coordinate (AB) at ($ (A)!0.5!(B) $);
        \coordinate (BC) at ($ (B)!0.5!(C) $);
        \coordinate (CA) at ($ (C)!0.5!(A) $);
        
        \draw (A) -- (B) -- (C) -- cycle;
        \draw (AB) -- (BC) -- (CA) -- cycle;
    
        \node at (0.25,0.144) {$\bigtriangleup_2^1$};
        \node at (0.5,0.577) {$\bigtriangleup_2^2$};
        \node at (0.75,0.144) {$\bigtriangleup_2^3$};
    
        \node[below] at (AB) {$2$};
        \node[right] at (BC) {$2$};
        \node[left] at (CA) {$2$};
\end{tikzpicture}
\end{aligned}
\]
This proves the statement for $n=2$.

Next we assume that $\Delta F_n \vert _{SG_n\backslash\{u_1,u_2,u_3\}}=\mathsf{id}_n \vert_{SG_n\backslash\{u_1,u_2,u_3\}}$ holds $n \in \N$, and we show that the statement holds for $n+1$.
We restrict again our attention to the upper subtriangle $\triangle_{n+1}^2$. Within $\triangle_{n+1}^2$, the distance function $d_{n+1}$ coincides with the graph distance to the top corner vertex $u_2$. Thus, on $\triangle_{n+1}^2$, the distance from the top corner can be depicted as in Figure \ref{fig:graph_distance}.

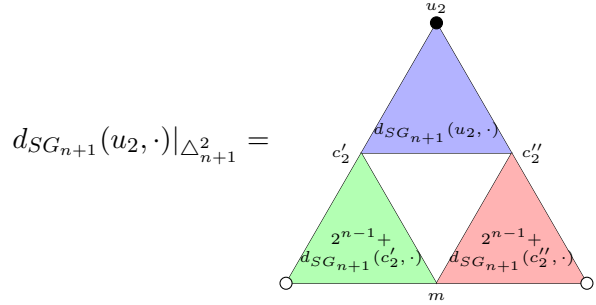
\begin{figure}[h]
    \begin{center}
    $d_{SG_{n+1}}(u_2, \cdot)\vert_{\bigtriangleup_{n+1}^2} = $
        \begin{tikzpicture}[scale=4, font = \tiny, baseline={(0,1.8)}]
            \coordinate (A) at (0,0);
            \coordinate (B) at (1,0);
            \coordinate (C) at (0.5,{sqrt(3)/2});
            
            \coordinate (AB) at ($ (A)!0.5!(B) $);
            \coordinate (BC) at ($ (B)!0.5!(C) $);
            \coordinate (CA) at ($ (C)!0.5!(A) $);
            
            \draw (A) -- (B) -- (C) -- cycle;
            \draw (AB) -- (BC) -- (CA) -- cycle;
            \fill[green!30] (A) -- (AB) -- (CA) -- cycle;
            \fill[red!30]   (B) -- (BC) -- (AB) -- cycle;
            \fill[blue!30]  (C) -- (CA) -- (BC) -- cycle;
            \draw[fill=white] (A) circle (0.02);    
            \draw[fill=white] (B) circle (0.02);
            \filldraw[black] (C) circle (0.02);
        
            \node[above] at (C) {$u_2$};
            \node[right] at (BC) {$c_2''$};
            \node[left] at (CA) {$c_2'$};
            \node[below] at (AB) {$m$};
            \node at (0.25,0.11)[align=center] {$2^{n-1} +$\\ $d_{SG_{n+1}}(c_2', \cdot)$};
            \node at (0.5,0.5) {$d_{SG_{n+1}}(u_2, \cdot)$};
            \node at (0.75,0.11)[align=center] {$2^{n-1} +$\\ $d_{SG_{n+1}}(c_2'', \cdot)$};
        \end{tikzpicture}
    \end{center}
    \caption{The values of the graph distance function in the subtriangle $\bigtriangleup_{n+1}^2$. }
    \label{fig:graph_distance}
\end{figure}
Since the Laplacian at $x$ depends solely on the function’s values at the neighbors of $x$, for all vertices in the blue triangle of Figure \ref{fig:graph_distance} that have no neighbors in the red or green triangles, the claim follows directly from the induction hypothesis.
Next, consider vertices lying in the green or red triangles that have no neighbors outside their respective triangles. Since within both the green and red regions the distance is uniformly shifted by the constant $2^{n-1}$, the claim again follows from the induction hypothesis. Therefore, for any vertex $x$ in the red, green, or blue triangle whose neighbors all lie in the same subtriangle as $x$, we have
$\Delta F_{n+1}(x)=\mathsf{id}_{n+1}(x)$.
Now we examine the values of the Laplacian at the cutpoints $c_2',c_2''$ and $m$ in $\triangle_{n+1}^2$ as shown in Figure \ref{fig:graph_distance}. For $c_2'$ we have
\begin{align*}
    \Delta d_{SG_{n+1}}(u_2,c_2') &=\sum _{y \sim c_2'}d_{SG_{n+1}}(u_2,c_2') - d_{SG_{n+1}}(u_2,y)\\
    &=4 \cdot 2^{n-1} - ((2^{n-1} - 1) + 2^{n-1} + 2 \cdot (2^{n-1}+1))
    = -1.
\end{align*}
By symmetry, for $c_2''$ it holds that $\Delta d_{SG_{n+1}}(u_2,c_2'')=-1$, and for $m$ similar computations lead to
\begin{equation*}
    \Delta d_{SG_{n+1}}(u_2,m)=4 \cdot 2^n - (2 \cdot (2^n - 1) + 2 \cdot 2^n) =2.
\end{equation*}
Hence, on $\triangle_{n+1}^2$ we have $\Delta F_{n+1}|_{\triangle_{n+1}^2}=\mathsf{id}_{n+1}|_{\triangle_{n+1}^2}$, and by 
rotational symmetry, the same holds on $\triangle_{n+1}^1$ and $\triangle_{n+1}^3$.
It remains to address the cutpoints $p_1, p_2, p_3$ of $SG_{n+1}$; see Figure \ref{fig:labels-for-proof} for the cutpoints of $SG_n$. By symmetry, the values of $\Delta d_{n+1}(\cdot)$ at these vertices agree and are equal to
\begin{align*}
    \Delta d_{n+1} (p_1) &= \sum_{v \sim p_1} d_{n+1}(p_1) - d_{n+1}(v)
    = 4 \cdot 2^n - (2 \cdot 2^n + 2 \cdot (2^n -1))
    = 2.
\end{align*}
This completes the argument, yielding
$\Delta_n F_n(x) = \mathsf{id}_n(x)$
for all $n \in \mathbb{N}$ and all $x \in V_{SG_n}\setminus\{u_1,u_2,u_3\}$. Define
$T_n = F_n - I_n$.
Then $\Delta_n T_n(x) = 0$ for all $x \in V_{SG_n}\setminus\{u_1,u_2,u_3\}$ and $T_n(u_i)=0$ for each $i\in\{1,2,3\}$. By the maximum	
(–minimum) principle for harmonic functions, $T_n(x)=0$ for all $x\in V_{SG_n}$, hence $I_n = F_n$, completing the proof.
\end{proof}

\section{Scaling limit of the sandpile identity}\label{sec:scaling}
We prove the three parts of Theorem \ref{thm:main} separately, beginning with \ref{I2}. This follows immediately from the decomposition of the sandpile identity established in Proposition \ref{prop:id-decomp}.

Let $d_{SG}(x,y)$ denote the path distance between points $x,y\in SG$, which can be obtained from the graph distance on $SG_n$ via
$$d_{SG}(x,y)=\lim_{n\rightarrow\infty}\frac{1}{2^n}d_{SG_n}(x_n,y_n),$$
where $x_n$ and $y_n$ are the vertices in $SG_n$ closest to $x$ and $y$ respectively. Write
$$\mathsf{d}(x)=\min\{\,d_{SG}(x,u_1),\, d_{SG}(x,u_2),\, d_{SG}(x,u_3)\,\}.$$

\begin{proposition}\label{prop:I2}
It holds that
    $$\lim_{n\rightarrow\infty}\frac{1}{2^n}\Bigl(g_n*\mathsf{id}_n-\frac{8}{3}h_n\Bigr)^\blacktriangle(x)=-\frac{1}{3}\mathsf{d}(x).$$
\end{proposition}
\begin{proof}
In view of Proposition \ref{prop:id-decomp}
$$g_n*\mathsf{id}_n(x)-\frac{8}{3}h_n(x)=-\frac{1}{3}d_n(x),$$
for all $x \in V_{SG_n}$, from which the claim follows by using
 $$\lim_{n\rightarrow\infty}\frac{1}{2^n}\Bigl(g_n*\mathsf{id}_n-\frac{8}{3}h_n\Bigr)^\blacktriangle(x)=\lim_{n\rightarrow\infty}-\frac{1}{3\cdot2^n}d_n^\blacktriangle(x)=-\frac{1}{3}\mathsf{d}(x).$$
\end{proof}

We now address statement \ref{I3} of Theorem \ref{thm:main}. The first step is to establish that the functions $h_n$ provide a good approximation to the integral of the Green function $G$ over the Sierpiński gasket. 

\begin{lemma}\label{lem:green-conv}
There exists a constant $c_{*}>0$ such that for all $x\in V_{*}$, it holds that
$$ \limsup_{n\rightarrow\infty}\bigg|3\cdot5^{n}\int_{SG}G(x,y)d\mu(y)-2\sum_{y\in V_{n}}g_{n}(x,y)\bigg|\le c_{*}. $$
\end{lemma}

\begin{proof}
First note that if $x\in V_{0}$, then $G(x,y)=0$ for all $y\in SG$, and $g_{n}(x,y)=0$ for all $n\in\mathbb{N}$ and all $y\in V_{n}$, making the lemma vacuously true.

Therefore we fix $x\in V_{*}\backslash V_{0}$. By definition, $x\in V_{n}$ for all sufficiently large $n\in\mathbb{N}$. Also, given $n\in\mathbb{N}$ we may partition $SG$ into balls of diameter $2^{-(n+1)}$ in the Euclidean distance centered at each $z\in V_{n}$, $SG=\bigcup_{z\in V_{n}}B(z,2^{-(n+1)})$. If $z\in V_{n}\backslash V_{0}$, then the closed ball $B(z,2^{-(n+1)})$ consists of two $(n+1)$-cells adjoined at $z$, and if $z\in V_{0}$, then $B(z,2^{-(n+1)})$ is the single $(n+1)$-cell containing the corner vertex $z$. As a result, under the standard self-similar probability measure $\mu$, the mass assigned to each ball is
\begin{align}\label{eq:volume}
\mu\big(B(z,2^{-(n+1)})\big) = \begin{cases} 2\cdot 3^{-(n+1)}, & \text{if } z\in V_{n}\backslash V_{0} \\ 3^{-(n+1)}, & \text{if } z\in V_{0} \end{cases}
\end{align}

Note that while two different balls may intersect, the set of their intersection has $\mu$-measure 0. We can thus express the integral of the Green's function as the sum of contributions from each ball of the said partition
\begin{align*}
\int_{SG}G(x,y)d\mu(y)&=\sum_{z\in V_{n}}\int_{B(z,2^{-(n+1)})}G(x,y)d\mu(y) \\
&=\sum_{z\in V_{n}\backslash V_{0}}\int_{B(z,2^{-(n+1)})}G(x,y)d\mu(y)+\sum_{z\in V_{0}}\int_{B(z,2^{-(n+1)})}G(x,y)d\mu(y) \\
&=:I_{bulk}+I_{boundary}.
\end{align*}
To analyze $I_{bulk}$, by the harmonic spline construction of the Green's function, the function $G(x,\cdot)$ is harmonic on $B(z,2^{-(n+1)})$ as long as $z\ne x$ and $z\notin V_{0}$. Applying Lemma \ref{lem:mean-value}, we get
\begin{align}\label{eq:green-bulk}
\frac{1}{\mu(B(z,2^{-(n+1)}))}\int_{B(z,2^{-(n+1)})}G(x,y)d\mu(y)=G(x,z) \quad \text{for every } z\in V_{n}\backslash V_{0},~z\ne x.
\end{align}
which together with Equations (\ref{eq:volume}) and (\ref{eq:green-bulk}) yields
\begin{align*}
I_{bulk}&=2\cdot3^{-(n+1)}\sum_{z\in V_{n}\backslash V_{0},z\neq x}G(x,z)+\int_{B(x,2^{-(n+1)})}G(x,y)d\mu(y) \\
&=2\cdot3^{-(n+1)}\sum_{z\in V_{n}}G(x,z)-2\cdot3^{-(n+1)}G(x,x)+\int_{B(x,2^{-(n+1)})}G(x,y)d\mu(y) \\
&=\frac{2}{3}\cdot\frac{1}{5^{n}}\sum_{z\in V_{n}}g_{n}(x,z)+\int_{B(x,2^{-(n+1)})}\big(G(x,y)-G(x,x)\big)d\mu(y).
\end{align*}
Above, we used the fact that $G(x,z)=0$ if one of $x$ and $z$ belongs to $V_{0}$, as well as Lemma \ref{lem:green-is-same} relating $G$ to the discrete Green function.
Turning now to the main expression in the Lemma, we can write
\begin{equation}\label{eq:main-expression}
\begin{aligned}
3\cdot5^{n}&\int_{SG}G(x,y)d\mu(y)-2\sum_{y\in V_{n}}g_{n}(x,y)=3\cdot5^{n}(I_{bulk}+I_{boundary})-2\sum_{y\in V_{n}}g_{n}(x,y)\\
&=3\cdot5^{n}\int_{B(x,2^{-(n+1)})}\big(G(x,y)-G(x,x)\big)d\mu(y)+3\cdot5^{n}\sum_{z\in V_{0}}\int_{B(z,2^{-(n+1)})}G(x,y)d\mu(y).
\end{aligned}
\end{equation}
We claim that this expression is uniformly bounded in $n\in\mathbb{N}$ and in $x\in V_{*}$ by a universal constant. To prove this, we need two results concerning the Green's function, each of which is used to bound each of the two terms in (\ref{eq:main-expression}).
The first result bounds the variation of the Green function on the balls $B(x,2^{-(n+1})$ for all $x\in V_*\backslash V_0$, $n\in\N$ and $y\in B(x,2^{-(n+1)})$ and this is part $(1)$ in Lemma \ref{lem:green-bounds},  showing that the first term in Equation (\ref{eq:main-expression}) is negative.
The second needed result bounds the decay of the Green function near the boundary $V_0$ and this is part $(2)$ of Lemma \ref{lem:green-bounds}, showing that the second term in (\ref{eq:main-expression}) is positive. Therefore the absolute value of \eqref{eq:main-expression} is bounded above by
\begin{align*}
&\max\bigg\{3\cdot5^{n}\int_{B(x,2^{-(n+1)})}\big(G(x,x)-G(x,y)\big)d\mu(y),~3\cdot5^{n}\sum_{z\in V_{0}}\int_{B(z,2^{-(n+1)})}G(x,y)d\mu(y)\bigg\} \\
&\le \max\bigg\{3\cdot5^{n}\cdot2\cdot3^{-(n+1)}\cdot C\left(\frac{3}{5}\right)^{n+1},~3\cdot5^{n}\cdot3\cdot3^{-(n+1)}\cdot C\left(\frac{3}{5}\right)^{n+1}\bigg\} \\
&=\max\bigg\{\frac{6}{5}C,~\frac{9}{5}C\bigg\}=\frac{9}{5}C,
\end{align*}
thereby proving the above claim and hence the Lemma.
\end{proof}

We now prove \ref{I3} in Theorem \ref{thm:main}.
\begin{proposition}\label{prop:I3}
For all $x\in V_*$ it holds that
 $$\lim_{n\rightarrow\infty}\frac{1}{2^n}\biggl(g_n*\mathsf{id}_n(x)-4\cdot5^n\int_{SG}G(x,y)d\mu(y)\biggr)=-\frac{1}{3}\mathsf{d}(x).$$
\end{proposition}
\begin{proof}
Let $x\in V_*$, then by Proposition \ref{prop:id-decomp}
\begin{align*}
\frac{1}{2^n}\biggl(g_n*\mathsf{id}_n(x)- & 4\cdot5^n\int_{SG}G(x,y)d\mu(y)\biggr)\\
&=-\frac{1}{3\cdot 2^n}d_n(x)+\frac{4}{3\cdot2^n}\bigg(2h_n(x)-3\cdot 5^n\int_{SG}G(x,y)d\mu(y)\bigg).
    \end{align*}
Lemma \ref{lem:green-conv} gives
    $$\frac{4}{3\cdot 2^n}\bigg|2h_n(x)-3\cdot 5^n\int_{SG}G(x,y)d\mu(y)\bigg|\leq c_{*}\frac{4}{3\cdot 2^n}\xrightarrow{n\rightarrow\infty}0,$$
which together with $\lim_{n}2^{-n}{d_{SG_n}(x)}=\mathsf{d}(x)$ yield the claim.
\end{proof}

Using our previous result, we can now directly prove \ref{I1} in Theorem \ref{thm:main} by using the already established \ref{I3}.

\begin{proposition}\label{prop:I1}
    The function $g_n*\mathsf{id}_n$ grows asymptotically as $5^n$, and for all $x\in SG$ it holds that
        \begin{align*}
            \lim_{n\rightarrow\infty}\frac{1}{5^n}(g_n*\mathsf{id}_n)^{\blacktriangle}(x)=4\int_{SG}G(x,y)d\mu(y).
        \end{align*}
\end{proposition}
\begin{proof}
Recall for every $n\in\N$ the definition of the piecewise-constant continuatione of $u^\blacktriangle:SG\rightarrow\R,u^\blacktriangle(x)=u(x_n)$, where $x_n\in V_n$ is the closest point to $x\in SG$ in the $n$-th iteration. Let us introduce the shorthand notation $\beta_{n}:SG\rightarrow\mathbb{R}$ given by
$$ \beta_{n}(x):=\frac{1}{5^{n}}(g_{n}*id_{n})^{\blacktriangle}(x)-4\int_{SG}G(x,y)d\mu(y),$$
which can be rewritten as
\begin{align*}
\beta_{n}(x)&=\left(\frac{1}{5^{n}}(g_{n}*id_{n})^{\blacktriangle}(x)-4\int_{SG}(G(\cdot,y))^{\blacktriangle}(x)d\mu(y)\right)+4\int_{SG}\big((G(\cdot,y))^{\blacktriangle}(x)-G(x,y)\big)d\mu(y) \\
&=:D_{n}(x)+D_{n}^{\prime}(x).
\end{align*}
Using Proposition \ref{prop:I3} and $\sup_{x\in SG}\mathsf{d}(x)=\frac{1}{2}$ we get
\begin{align*}
\limsup_{n\rightarrow\infty}|D_{n}(x)|&\le \limsup_{n\rightarrow\infty}\frac{1}{2^{n}}\left|(g_{n}*id_{n})^{\blacktriangle}(x)-4\cdot5^{n}\int_{SG}(G(\cdot,y))^{\blacktriangle}(x)d\mu(y)\right|\cdot \limsup_{n\rightarrow\infty}\frac{2^{n}}{5^{n}} = 0.
\end{align*}
On the other hand, by Lemma \ref{lem:green-variation}, for every $x\in SG$ we have
$$ \big|(G(\cdot,y))^{\blacktriangle}(x)-G(x,y)\big|=|G(x_{n},y)-G(x,y)|\le R_{eff}(x_{n},x)\xrightarrow{n\rightarrow\infty}0,$$
which together with the dominated convergence theorem yields $\lim_{n\rightarrow\infty}D_{n}^{\prime}(x)=0$. Hence it holds $\lim_{n\rightarrow\infty}\beta_{n}(x)=0$ and this completes the proof.
\end{proof}

\begin{remark}\normalfont
In this work, we examined the sequence of identity elements in the sandpile group on the Sierpiński gasket approximation graphs and showed that, after smoothing via convolution with the Green operator, the identity admits both first- and second-order scaling limits. The first-order limit is the integral of the Green function on the Sierpiński gasket, while the second-order limit is the distance to the gasket’s corner vertices.
An intriguing question is whether the emergence of the distance function in the Sierpiński gasket case is incidental and fails for other graph sequences, or whether a decomposition analogous to Proposition \ref{prop:id-decomp} also exists more broadly. Identifying such a structure for sandpile identities on more general graphs would allow to analyze scaling limits beyond the Sierpiński gasket, across a wider range of state spaces.

\end{remark}

\textbf{Acknowledgments.} 
The authors are very grateful to the anonymous referee for an exceptionally thorough and constructive review. We specifically acknowledge the referee’s direct mathematical contributions to this revision, including correcting Lemma 4.1, suggesting a cleaner proof of the limit-distance identity, and formulating technical lemmas that substantially strengthened the fractal analysis.

\textbf{Funding information.} The research of E. Sava-Huss and J. Überbacher was funded by Austrian Science Fund (FWF) 10.55776/PAT3123425. For open access purposes, the authors have applied a CC BY public copyright license to any author-accepted manuscript version arising from this submission.

\bibliography{literature}

@article {FHKS24,
    AUTHOR = {Freiberg, Uta and Heizmann, Nico and Kaiser, Robin and
              Sava-Huss, Ecaterina},
     TITLE = {Internal aggregation models with multiple sources and obstacle
              problems on {S}ierpi\'nski gaskets},
   JOURNAL = {J. Fractal Geom.},
  FJOURNAL = {Journal of Fractal Geometry. Mathematics of Fractals and
              Related Topics},
    VOLUME = {11},
      YEAR = {2024},
    NUMBER = {1-2},
     PAGES = {111--160},
      ISSN = {2308-1309,2308-1317},
   MRCLASS = {31E05 (28A80 31C20 35R35 60J10 60J45)},
  MRNUMBER = {4751210},
MRREVIEWER = {Peter\ R.\ Massopust},
       DOI = {10.4171/jfg/141},
       URL = {https://doi-org.tum-eaccess.de/10.4171/jfg/141},
}

@book {K01,
    AUTHOR = {Kigami, Jun},
     TITLE = {Analysis on fractals},
    SERIES = {Cambridge Tracts in Mathematics},
    VOLUME = {143},
 PUBLISHER = {Cambridge University Press, Cambridge},
      YEAR = {2001},
     PAGES = {viii+226},
      ISBN = {0-521-79321-1},
   MRCLASS = {28A80 (31C20 31C25 35J05 35K05)},
  MRNUMBER = {1840042},
MRREVIEWER = {Volker\ Metz},
       DOI = {10.1017/CBO9780511470943},
       URL = {https://doi-org.tum-eaccess.de/10.1017/CBO9780511470943},
}

@book {S06,
    AUTHOR = {Strichartz, Robert S.},
     TITLE = {Differential equations on fractals},
      NOTE = {A tutorial},
 PUBLISHER = {Princeton University Press, Princeton, NJ},
      YEAR = {2006},
     PAGES = {xvi+169},
      ISBN = {978-0-691-12731-6; 0-691-12731-X},
   MRCLASS = {35-02 (28A80 35J05)},
  MRNUMBER = {2246975},
MRREVIEWER = {Peter\ R.\ Massopust},
}

@article {JLP19,
    AUTHOR = {Jerison, Daniel C. and Levine, Lionel and Pike, John},
     TITLE = {Mixing time and eigenvalues of the abelian sandpile {M}arkov
              chain},
   JOURNAL = {Trans. Amer. Math. Soc.},
  FJOURNAL = {Transactions of the American Mathematical Society},
    VOLUME = {372},
      YEAR = {2019},
    NUMBER = {12},
     PAGES = {8307--8345},
      ISSN = {0002-9947,1088-6850},
   MRCLASS = {60J10 (05C50 82C20)},
  MRNUMBER = {4029698},
MRREVIEWER = {Arvind\ Ayyer},
       DOI = {10.1090/tran/7585},
       URL = {https://doi-org.tum-eaccess.de/10.1090/tran/7585},
}

@article {HJL19,
    AUTHOR = {Hough, Robert D. and Jerison, Daniel C. and Levine, Lionel},
     TITLE = {Sandpiles on the square lattice},
   JOURNAL = {Comm. Math. Phys.},
  FJOURNAL = {Communications in Mathematical Physics},
    VOLUME = {367},
      YEAR = {2019},
    NUMBER = {1},
     PAGES = {33--87},
      ISSN = {0010-3616,1432-0916},
   MRCLASS = {82C22 (60K35)},
  MRNUMBER = {3933404},
       DOI = {10.1007/s00220-019-03408-5},
       URL = {https://doi-org.tum-eaccess.de/10.1007/s00220-019-03408-5},
}

@article {BHJ17,
    AUTHOR = {Bhupatiraju, Sandeep and Hanson, Jack and J\'arai, Antal A.},
     TITLE = {Inequalities for critical exponents in {$d$}-dimensional
              sandpiles},
   JOURNAL = {Electron. J. Probab.},
  FJOURNAL = {Electronic Journal of Probability},
    VOLUME = {22},
      YEAR = {2017},
     PAGES = {Paper No. 85, 51},
      ISSN = {1083-6489},
   MRCLASS = {60K35 (82C22)},
  MRNUMBER = {3718713},
MRREVIEWER = {Xin\ Sun},
       DOI = {10.1214/17-EJP111},
       URL = {https://doi-org.tum-eaccess.de/10.1214/17-EJP111},
}

@article {JW14,
    AUTHOR = {J\'arai, Antal A. and Werning, Nicol\'as},
     TITLE = {Minimal configurations and sandpile measures},
   JOURNAL = {J. Theoret. Probab.},
  FJOURNAL = {Journal of Theoretical Probability},
    VOLUME = {27},
      YEAR = {2014},
    NUMBER = {1},
     PAGES = {153--167},
      ISSN = {0894-9840,1572-9230},
   MRCLASS = {60K35 (82C20)},
  MRNUMBER = {3174221},
MRREVIEWER = {Oriane\ Blondel},
       DOI = {10.1007/s10959-012-0446-z},
       URL = {https://doi-org.tum-eaccess.de/10.1007/s10959-012-0446-z},
}

@article {AJ04,
    AUTHOR = {Athreya, Siva R. and J\'arai, Antal A.},
     TITLE = {Infinite volume limit for the stationary distribution of
              abelian sandpile models},
   JOURNAL = {Comm. Math. Phys.},
  FJOURNAL = {Communications in Mathematical Physics},
    VOLUME = {249},
      YEAR = {2004},
    NUMBER = {1},
     PAGES = {197--213},
      ISSN = {0010-3616,1432-0916},
   MRCLASS = {82C22 (60K35 82B31)},
  MRNUMBER = {2077255},
MRREVIEWER = {Ellen\ Saada},
       DOI = {10.1007/s00220-004-1080-0},
       URL = {https://doi-org.tum-eaccess.de/10.1007/s00220-004-1080-0},
}

@article {PS13,
    AUTHOR = {Pegden, Wesley and Smart, Charles K.},
     TITLE = {Convergence of the {A}belian sandpile},
   JOURNAL = {Duke Math. J.},
  FJOURNAL = {Duke Mathematical Journal},
    VOLUME = {162},
      YEAR = {2013},
    NUMBER = {4},
     PAGES = {627--642},
      ISSN = {0012-7094,1547-7398},
   MRCLASS = {60K35 (35D40 35J60 35R35 60J60)},
  MRNUMBER = {3039676},
MRREVIEWER = {Antal\ A.\ J\'arai},
       DOI = {10.1215/00127094-2079677},
       URL = {https://doi-org.tum-eaccess.de/10.1215/00127094-2079677},
}

@article {LPS16,
    AUTHOR = {Levine, Lionel and Pegden, Wesley and Smart, Charles K.},
     TITLE = {Apollonian structure in the {A}belian sandpile},
   JOURNAL = {Geom. Funct. Anal.},
  FJOURNAL = {Geometric and Functional Analysis},
    VOLUME = {26},
      YEAR = {2016},
    NUMBER = {1},
     PAGES = {306--336},
      ISSN = {1016-443X,1420-8970},
   MRCLASS = {60K35 (28A80 35D40 35R02 35R35)},
  MRNUMBER = {3494492},
MRREVIEWER = {Max\ Fathi},
       DOI = {10.1007/s00039-016-0358-7},
       URL = {https://doi-org.tum-eaccess.de/10.1007/s00039-016-0358-7},
}

@article {KSHW24,
    AUTHOR = {Kaiser, Robin and Sava-Huss, Ecaterina and Wang, Yuwen},
     TITLE = {Abelian sandpiles on {S}ierpi\'nski gasket graphs},
   JOURNAL = {Electron. J. Combin.},
  FJOURNAL = {Electronic Journal of Combinatorics},
    VOLUME = {31},
      YEAR = {2024},
    NUMBER = {1},
     PAGES = {Paper No. 1.6, 23},
      ISSN = {1077-8926},
   MRCLASS = {05C81 (20K01 31C20 60J10)},
  MRNUMBER = {4695551},
MRREVIEWER = {David\ J.\ Aldous},
       DOI = {10.37236/11520},
       URL = {https://doi.org/10.37236/11520},
}

@article {J18,
    AUTHOR = {J\'arai, Antal A.},
     TITLE = {Sandpile models},
   JOURNAL = {Probab. Surv.},
  FJOURNAL = {Probability Surveys},
    VOLUME = {15},
      YEAR = {2018},
     PAGES = {243--306},
      ISSN = {1549-5787},
   MRCLASS = {60K35 (82C24)},
  MRNUMBER = {3857602},
MRREVIEWER = {Daniel\ Boivin},
       DOI = {10.1214/14-PS228},
       URL = {https://doi.org/10.1214/14-PS228},
}

@article {BP88,
    AUTHOR = {Barlow, Martin T. and Perkins, Edwin A.},
     TITLE = {Brownian motion on the {S}ierpi\'nski gasket},
   JOURNAL = {Probab. Theory Related Fields},
  FJOURNAL = {Probability Theory and Related Fields},
    VOLUME = {79},
      YEAR = {1988},
    NUMBER = {4},
     PAGES = {543--623},
      ISSN = {0178-8051,1432-2064},
   MRCLASS = {60J60 (60J25 60J65)},
  MRNUMBER = {966175},
MRREVIEWER = {F.\ B.\ Knight},
       DOI = {10.1007/BF00318785},
       URL = {https://doi.org/10.1007/BF00318785},
}

@article {BTW88,
    AUTHOR = {Bak, Per and Tang, Chao and Wiesenfeld, Kurt},
     TITLE = {Self-organized criticality},
   JOURNAL = {Phys. Rev. A (3)},
  FJOURNAL = {Physical Review. A. Third Series},
    VOLUME = {38},
      YEAR = {1988},
    NUMBER = {1},
     PAGES = {364--374},
      ISSN = {1050-2947,1094-1622},
   MRCLASS = {58F13 (82A25 92A05)},
  MRNUMBER = {949160},
       DOI = {10.1103/PhysRevA.38.364},
       URL = {https://doi.org/10.1103/PhysRevA.38.364},
}

@article{BTW87,
  title = {Self-organized criticality: An explanation of the 1/f noise},
  author = {Bak, Per and Tang, Chao and Wiesenfeld, Kurt},
  journal = {Phys. Rev. Lett.},
  volume = {59},
  issue = {4},
  pages = {381--384},
  numpages = {0},
  year = {1987},
  month = {Jul},
  publisher = {American Physical Society},
  doi = {10.1103/PhysRevLett.59.381},
  url = {https://link.aps.org/doi/10.1103/PhysRevLett.59.381}
}

@article {D90,
    AUTHOR = {Dhar, Deepak},
     TITLE = {Self-organized critical state of sandpile automaton models},
   JOURNAL = {Phys. Rev. Lett.},
  FJOURNAL = {Physical Review Letters},
    VOLUME = {64},
      YEAR = {1990},
    NUMBER = {14},
     PAGES = {1613--1616},
      ISSN = {0031-9007},
   MRCLASS = {82A68 (82A60)},
  MRNUMBER = {1044086},
       DOI = {10.1103/PhysRevLett.64.1613},
       URL = {https://doi.org/10.1103/PhysRevLett.64.1613},
}

@article {DRSV95,
    AUTHOR = {Dhar, D. and Ruelle, P. and Sen, S. and Verma, D.-N.},
     TITLE = {Algebraic aspects of abelian sandpile models},
   JOURNAL = {J. Phys. A},
  FJOURNAL = {Journal of Physics. A. Mathematical and General},
    VOLUME = {28},
      YEAR = {1995},
    NUMBER = {4},
     PAGES = {805--831},
      ISSN = {0305-4470,1751-8121},
   MRCLASS = {82B20 (20K01 82B03)},
  MRNUMBER = {1326322},
MRREVIEWER = {Anatoliy\ Yu.\ Zakharov},
       URL = {http://stacks.iop.org/0305-4470/28/805},
}

@article {CKF20,
    AUTHOR = {Chen, Joe P. and Kudler-Flam, Jonah},
     TITLE = {Laplacian growth and sandpiles on the {S}ierpi\'nski gasket:
              limit shape universality and exact solutions},
   JOURNAL = {Ann. Inst. Henri Poincar\'e{} D},
  FJOURNAL = {Annales de l'Institut Henri Poincar\'e{} D. Combinatorics,
              Physics and their Interactions},
    VOLUME = {7},
      YEAR = {2020},
    NUMBER = {4},
     PAGES = {585--664},
      ISSN = {2308-5827,2308-5835},
   MRCLASS = {82C24 (05C20 05C81 28A80 31C45 37B15 60K05)},
  MRNUMBER = {4182776},
MRREVIEWER = {Manuel\ Mor\'an},
       DOI = {10.4171/aihpd/95},
       URL = {https://doi.org/10.4171/aihpd/95},
}

@incollection{KSH26,
  author    = {Robin Kaiser and Ecaterina Sava-Huss},
  title     = {Scaling limit of the sandpile identity element on the {S}ierpinski gasket},
  booktitle = {From Classical Analysis to Analysis on Fractals: A Tribute to Robert Strichartz},
  series    = {Applied and Numerical Harmonic Analysis},
  volume    = {2},
  publisher = {Birkh\"auser},
  year      = {2026},
  url={https://arxiv.org/abs/2308.12183}, 
}

@incollection {LBR02,
    AUTHOR = {Le Borgne, Yvan and Rossin, Dominique},
     TITLE = {On the identity of the sandpile group},
      NOTE = {LaCIM 2000 Conference on Combinatorics, Computer Science and
              Applications (Montreal, QC)},
   JOURNAL = {Discrete Math.},
  FJOURNAL = {Discrete Mathematics},
    VOLUME = {256},
      YEAR = {2002},
    NUMBER = {3},
     PAGES = {775--790},
      ISSN = {0012-365X,1872-681X},
   MRCLASS = {82C20 (37B15 60G18 68Q80)},
  MRNUMBER = {1935788},
MRREVIEWER = {E.\ J.\ Janse van Rensburg},
       DOI = {10.1016/S0012-365X(02)00347-3},
       URL = {https://doi.org/10.1016/S0012-365X(02)00347-3},
}

@article {FLP10,
    AUTHOR = {Fey, Anne and Levine, Lionel and Peres, Yuval},
     TITLE = {Growth rates and explosions in sandpiles},
   JOURNAL = {J. Stat. Phys.},
  FJOURNAL = {Journal of Statistical Physics},
    VOLUME = {138},
      YEAR = {2010},
    NUMBER = {1-3},
     PAGES = {143--159},
      ISSN = {0022-4715,1572-9613},
   MRCLASS = {82C20 (60K35 82C43)},
  MRNUMBER = {2594895},
MRREVIEWER = {Aernout\ C. D. van Enter},
       DOI = {10.1007/s10955-009-9899-6},
       URL = {https://doi.org/10.1007/s10955-009-9899-6},
}

@article {JL07,
    AUTHOR = {J\'arai, A. A. and Lyons, R.},
     TITLE = {Ladder sandpiles},
   JOURNAL = {Markov Process. Related Fields},
  FJOURNAL = {Markov Processes and Related Fields},
    VOLUME = {13},
      YEAR = {2007},
    NUMBER = {3},
     PAGES = {493--518},
      ISSN = {1024-2953},
   MRCLASS = {82C20 (60C05 60K35)},
  MRNUMBER = {2357385},
}

@article {ENP23,
    AUTHOR = {Eckmann, Jean-Pierre and Nagnibeda, Tatiana and Perriard,
              Aymeric},
     TITLE = {Abelian sandpiles on cylinders},
   JOURNAL = {J. Phys. A},
  FJOURNAL = {Journal of Physics. A. Mathematical and Theoretical},
    VOLUME = {56},
      YEAR = {2023},
    NUMBER = {17},
     PAGES = {Paper No. 175001, 11},
      ISSN = {1751-8113,1751-8121},
   MRCLASS = {82C20 (60K35)},
  MRNUMBER = {4571436},
       DOI = {10.1088/1751-8121/acc435},
       URL = {https://doi.org/10.1088/1751-8121/acc435},
}

@article {M22,
    AUTHOR = {Melchionna, Andrew},
     TITLE = {The sandpile identity element on an ellipse},
   JOURNAL = {Discrete Contin. Dyn. Syst.},
  FJOURNAL = {Discrete and Continuous Dynamical Systems},
    VOLUME = {42},
      YEAR = {2022},
    NUMBER = {8},
     PAGES = {3709--3732},
      ISSN = {1078-0947,1553-5231},
   MRCLASS = {35Q70 (05C57 31C20 37B15)},
  MRNUMBER = {4447555},
       DOI = {10.3934/dcds.2022029},
       URL = {https://doi.org/10.3934/dcds.2022029},
}

@article {Kig12,
    AUTHOR = {Kigami, Jun},
     TITLE = {Resistance forms, quasisymmetric maps and heat kernel
              estimates},
   JOURNAL = {Mem. Amer. Math. Soc.},
  FJOURNAL = {Memoirs of the American Mathematical Society},
    VOLUME = {216},
      YEAR = {2012},
    NUMBER = {1015},
     PAGES = {vi+132},
      ISSN = {0065-9266,1947-6221},
      ISBN = {978-0-8218-5299-6},
   MRCLASS = {30L10 (28A80 31C25 35K05 35K08 60J45)},
  MRNUMBER = {2919892},
MRREVIEWER = {Leonid\ V.\ Kovalev},
       DOI = {10.1090/S0065-9266-2011-00632-5},
       URL = {https://doi-org.tum-eaccess.de/10.1090/S0065-9266-2011-00632-5},
}
\bibliographystyle{alpha}

\textsc{Robin Kaiser}, Departement of Mathematics, CIT, Technische Universität München, Boltzmannstr. 3, D-85748 Garching bei München, Germany. \texttt{ro.kaiser@tum.de}

\textsc{Ecaterina Sava-Huss}, Universität Innsbruck, Institut für Mathematik, Technikerstraße 13,
A-6020 Innsbruck, Austria. \texttt{Ecaterina.Sava-Huss@uibk.ac.at}

\textsc{Julia Überbacher}, Universität Innsbruck, Institut für Mathematik, Technikerstraße 13,
A-6020 Innsbruck, Austria. \texttt{Julia.Ueberbacher@uibk.ac.at}
\end{document}